\documentclass[journal]{IEEEtran}
\usepackage{amsmath,amssymb,amsthm}
\usepackage{graphicx,xcolor}
\usepackage{cite}
\usepackage[colorlinks=true,linkcolor=blue,citecolor=blue,urlcolor=blue]{hyperref}

\graphicspath{{./},{./figures/}}

\newtheorem{lemma}{Lemma}
\newtheorem{proposition}{Proposition}

\newtheorem{remark}{Remark}

\newcommand{\R}{\mathbb{R}}

\newcommand{\Spp}{\mathbb{S}^{n}_{++}}
\newcommand{\range}{\operatorname{\mathcal R}}
\newcommand{\Null}{\operatorname{\mathcal N}}
\newcommand{\GLp}{\mathrm{GL}^{+}}

\title{\huge Addendum to ``Optimal Steering of a Linear Stochastic System
to a Final Probability Distribution, Part II''}
\author{Yongxin Chen$^\star$, Tryphon T.\ Georgiou$^\dagger$, and Michele Pavon$^\ddagger$\thanks{This work was supported in part by the NSF under grants ECCS-2347357 and 1942523,
the AFOSR under FA9550-24-1-0278, the ARO under W911NF-22-1-0292\\
$^\star$School of Aerospace Engineering, Georgia Institute of Technology, Atlanta, GA 30332 USA
(yongchen@gatech.edu).\\
$^\dagger$Mechanical and Aerospace Engineering, University of California, Irvine, Irvine,
CA 92697 USA (tryphon@uci.edu).\\
$^\ddagger$Universit\`a di Padova, Padova, 35121 Italy (pavon@math.unipd.it).
}}
\date{}

\begin{document}
\maketitle
\thispagestyle{empty}
\pagestyle{empty}
\nocite{ChenGeorgiouPavonPartI}

\begin{abstract}
We hereby fill in a gap in an induction argument used in the proof of
 \cite[Theorem~3]{ChenGeorgiouPavonPartII} on state-covariance controllability of linear stochastic systems.  The mathematical statement  is true.  The missing point in the proof concerns the assertion that, in the canonical
single-input case, the control can be assigned arbitrary endpoint values.
We amend the argument by proving a slightly stronger induction statement
that permits arbitrary finite endpoint jets of the control. An independent approach to homogeneous state-covariance controllability was provided in \cite{AbdelgalilGeorgiou2025} as a corollary of strong controllability results on $\GLp(n,\mathbb R)$, and shown herein that it extends to the case of additive noise as well.
\end{abstract}

\section{Introduction}

The steering of state statistics of stochastic dynamical systems--often
referred to as covariance control or covariance steering--was developed
systematically in \cite{ChenGeorgiouPavonPartI,ChenGeorgiouPavonPartII}.
These works substantially extended the earlier program of steady-state
covariance assignment initiated by Skelton and collaborators
\cite{HotzSkelton1987}, by addressing finite-horizon steering between
prescribed probability distributions. Part~II of this two-part contribution, dealt with  control and stochastic-disturbances entering the system dynamics through distinct channels.

In this development, \cite[Theorem~3]{ChenGeorgiouPavonPartII} asserts that the differential Lyapunov system 
\begin{equation}
  \dot\Sigma(t)
  =A\Sigma(t)+\Sigma(t) A^{\mathsf T}
   +BU(t)^{\mathsf T}+U(t)B^{\mathsf T}+Q,
  \label{eq:lyap}
\end{equation}
with $\Sigma(t)\in\mathbb S^n$, constant matrices $(A,B)\in\R^{n\times n}\times\R^{n\times m}$ and $Q\in\mathbb S^n$, $Q\succeq0$, and control $U(t)\in\R^{n\times m}$ 
is controllable if and only if the underlying pair $(A,B)$
is controllable, and moreover, that for any two 
endpoint covariances $\Sigma_0,\Sigma_T\succ 0$ and any $Q\succeq0$, there exists a smooth control producing a trajectory $\{\Sigma(t) \mid t\in[0,T]\}$
that remains positive definite throughout the prescribed time interval. Here $\mathbb S^n$ denotes the space of real symmetric
$n\times n$ matrices, and $\mathbb S_{++}^n$ its positive-definite cone.
Recently, Fengjiao Liu and Panagiotis Tsiotras kindly brought to our attention a gap in the induction used
in the published proof. Specifically, the argument invokes the ability to
prescribe endpoint values of an auxiliary control, without establishing
that this property is preserved by the induction step. The statement of
the theorem is nevertheless correct.

Well before this gap was brought to our attention, Abdelgalil and Georgiou \cite{AbdelgalilGeorgiou2025} 
had developed an independent route that establishes the controllability of the Lyapunov dynamics from a stronger
controllability result for state-transition dynamics that they obtained.
Specifically, under the Kalman rank
controllability condition on $(A,B)$, the state-transition dynamics
\begin{equation}
  \dot\Phi(t)=(A+BK(t))\Phi(t),\qquad \Phi(0)=I,
  \label{eq:transition-intro}
\end{equation}
are strongly controllable on $\GLp(n,\mathbb R)$: for every $T>0$ and every
$\Phi_T\in\GLp(n,\mathbb R)$, there exists a piece-wise smooth state-feedback gain
$K(\cdot)$ such that $\Phi(T)=\Phi_T$ \cite[Theorem 3]{AbdelgalilGeorgiou2025}. The proof is direct and proceeds by
factoring the prescribed terminal transition matrix into factors similar to symmetric positive-definite matrices, which serve as successive motion primitives.
The existence argument in \cite{AbdelgalilGeorgiou2025} rests on classical
factorization results of Ballantine, while a companion paper
\cite{AbdelgalilGeorgiouFactorization2026} provides an explicit algorithm
for constructing such factors and, consequently, a piecewise-smooth feedback protocol, produced by a method that allows control of the condition number of the factors.

The passage from the controllability of the state-transition dynamics
\eqref{eq:transition-intro} to that of the homogeneous Lyapunov dynamics is
elementary. Given $\Sigma_0,\Sigma_T\succ0$, choose, for example,
\[
  \Phi_T=\Sigma_T^{1/2}\Sigma_0^{-1/2},
\]
and let $\Phi(\cdot)$ be a trajectory of
\eqref{eq:transition-intro} joining $I$ to $\Phi_T$. Then
\[
  \Sigma(t)=\Phi(t)\Sigma_0\Phi(t)^{\mathsf T}\succ 0
\]
satisfies the boundary conditions as well as \eqref{eq:lyap} with $Q=0$, by taking $U(t)=\Sigma(t)K(t)^\top$ with the feedback gain $K(\cdot)$ as in \cite{AbdelgalilGeorgiou2025}. 

The addition of a prescribed noise covariance $Q\succeq0$ does not alter
the fact that the reachable set of covariances
\[
\mathcal R_T^Q(\Sigma_0):=\{\Sigma(T)\mid \eqref{eq:lyap} \mbox{ holds, and }\Sigma(t)\succ 0, \;\forall  t\in[0,T]\}
\]
from $\Sigma(0)=\Sigma_0$, coincides with the open cone $\mathbb S_{++}^n$ of positive definite $n\times n$ matrices, which is already $\mathcal R_T^0(\Sigma_0)$ from \cite{AbdelgalilGeorgiou2025}. We first include a short argument below (Section \ref{appendix}) that builds a proof of the controllability assertion for the covariance dynamics based on \cite{AbdelgalilGeorgiou2025}.
We then directly repair the original induction argument of \cite{ChenGeorgiouPavonPartII} by supplying the missing steps. These, in addition, recover the smoothness assertion in
\cite{ChenGeorgiouPavonPartII}: it produces a smooth
matrix-valued input $U(\cdot)$, and thereby a smooth state-feedback gain,
while accommodating an arbitrary term $Q\succeq0$.
We conclude by noting that Liu and Tsiotras have provided a detailed study of the
covariance controllability of time-varying linear systems in 
\cite{LiuTsiotras2026}.

\subsection{Reduction to the homogeneous case}\label{appendix}

The homogeneous steering result of \cite{AbdelgalilGeorgiou2025}
extends to \eqref{eq:lyap} with any prescribed constant $Q\succeq0$.
To see this, define
\[
  \Gamma(\tau):=\int_0^\tau e^{sA}Qe^{sA^{\mathsf T}}\,ds\succeq0.
\]
Then $\Gamma(0)=0$ and
$\dot\Gamma(\tau)=A\Gamma(\tau)+\Gamma(\tau)A^{\mathsf T}+Q$.
Then, given $T>0$ and $\Sigma_0,\Sigma_T\succ0$, choose
$\epsilon\in(0,T)$ sufficiently small that
\[
  \Sigma_T-\Gamma(\epsilon)\succ0,
\]
which is possible since $\Gamma(\epsilon)\to 0$ as
$\epsilon\to 0$.
Take $U(t)=0$ for $t\in[0,T-\epsilon)$, and denote the solution to \eqref{eq:lyap}
 at $T-\epsilon$ by $M=\Sigma(T-\epsilon)$; clearly $M\succ0$.
On the remaining interval, reset time to $\tau=t-(T-\epsilon)$.
By homogeneous controllability, there is a positive-definite trajectory $S(\tau)$ for $\tau\in[0,\epsilon]$, with a piecewise-smooth input $U_1(\tau)$ satisfying
\[
  \dot S(\tau)=AS(\tau)+S(\tau)A^{\mathsf T}+BU_1^{\mathsf T}(\tau)+U_1(\tau)B^{\mathsf T},
\]
\[
  S(0)=M,\quad S(\epsilon)=\Sigma_T-\Gamma(\epsilon).
\]
Set $\Sigma(t):=S(t-(T-\epsilon))+\Gamma(t-(T-\epsilon))$ over $[T-\epsilon,T]$. Adding the 
 differential equations for $S$ and $\Gamma$ shows that $\Sigma(t)$ satisfies \eqref{eq:lyap}
with
\[
U(t)=
\begin{cases}
0, & 0\leq t<T-\epsilon,\\
U_1(t-(T-\epsilon)), & T-\epsilon\leq t\leq T.
\end{cases}
\]
Evidently, $\Sigma(t)\succ0$
throughout, while
\[
 \Sigma(0)=\Sigma_0,\quad \Sigma(T-\epsilon)=M,\quad \Sigma(T)=\Sigma_T.
\]
Therefore, concatenating the controlled trajectory over $[T-\epsilon, T]$ with the initial uncontrolled segment
readily proves positive-definite steering over $[0,T]$.
Thus, the result of \cite{AbdelgalilGeorgiou2025} extends to prescribed
additive noise with piecewise-smooth controls. The corrected induction
below additionally establishes smooth steering.

\section{The gap in the original induction argument}

We return to the assertion of \cite[Theorem~3]{ChenGeorgiouPavonPartII}, 
that provided $(A,B)$ is controllable, any two matrices
$\Sigma_0,\Sigma_T\succ0$ can be joined over any prescribed interval
$[0,T]$ by a smooth control $U(\cdot)$ in such a way that
$\Sigma(t)\succ0$ for all $t\in[0,T]$.

The argument in \cite{ChenGeorgiouPavonPartII} first treats the controllable canonical single-input
pair and then invokes Heymann's lemma and feedback equivalence to reduce
the general case to it.  In the canonical case, the proof uses induction
on the dimension and claims, as part of the induction hypothesis, that
the control may be assigned arbitrary values at $t=0$ and $t=T$.
That strengthened assertion is needed because, invoking the block
partition
\begin{equation}
  \Sigma(t)=
  \begin{bmatrix}
    \Sigma_1(t) & \sigma_2(t)\\
    \sigma_2(t)^{\mathsf T} & \sigma_3(t)
  \end{bmatrix},
  \label{eq:blockSigma}
\end{equation}
the vector $\sigma_2(t)$ plays the role of the control for the
$(n-1)$-dimensional problem, while its endpoint values are fixed by
$\Sigma_0$ and $\Sigma_T$.

The gap is that the identity 
\[
  \range\bigl(U\mapsto BU^{\mathsf T}+UB^{\mathsf T}\bigr)
  =
  \Null\bigl(Y\mapsto \Pi_{\range(B)^\perp}
  Y\Pi_{\range(B)^\perp}\bigr),
\]
between range $\range(\cdot)$ and null space $\Null(\cdot)$,
only guarantees the existence of a control producing a prescribed
admissible tangent direction.  It does not imply that the control can
subsequently be assigned arbitrary endpoint values.  In fact, for the
canonical choice $B=e_n$, the map
\[
  u\longmapsto e_nu^{\mathsf T}+ue_n^{\mathsf T}
\]
is injective.  Thus, once $\Sigma(\cdot)$ is fixed, the control is fixed
as well.  The endpoint interpolation property therefore has to be built
into the induction itself.

\section{The amended canonical induction}

We begin with the following elementary lemma on interpolating boundary values of a positive function and a number of its derivatives (jets) at the end-points of an interval.

\begin{lemma}
\label{lem:positive-jets}
Let $r\geq0$.  Suppose that numbers
\[
  d_0^{(j)},\ d_T^{(j)},\qquad j=0,\ldots,r,
\]
are prescribed, with $d_0^{(0)}>0$ and $d_T^{(0)}>0$.
Then there exists $d\in C^\infty([0,T])$ such that $d(t)>0$ for $t\in[0,T]$ with boundary values
\[
  d^{(j)}(0)=d_0^{(j)},\;
  d^{(j)}(T)=d_T^{(j)}
\]
for $j=0,\ldots,r$, for the function and its derivatives.
\end{lemma}

\begin{proof}
Write $d=e^h$.  Since $d_0^{(0)}$ and $d_T^{(0)}$ are positive, the
prescribed jets of $d$ determine, recursively, unique endpoint jets of
$h=\log d$.  A Hermite interpolation polynomial can be chosen with
those endpoint jets.  Setting $d=e^h$ gives the conclusion.
\end{proof}

Now let $n\geq 1$, $e_n$ denote the $n$th standard basis vector
of $\mathbb R^n$, and define recursively for $n\geq2$
\begin{equation}
A_n=
\begin{bmatrix}
A_{n-1} & b_{n-1}\\
0 & 0
\end{bmatrix},
\qquad
B_n=e_n,
  \label{eq:canonical-pair}
\end{equation}
where $A_1=0$ and $b_{n-1}=e_{n-1}\in\mathbb R^{n-1}$.
Thus, $A_n$ is the standard (companion) shift matrix and $(A_n,B_n)$ is
the standard controllable canonical pair.
We prove a stronger statement than is required for a canonical steering with prescribed jets (values of derivatives) for the control at the two endpoints.

\begin{proposition}
\label{prop:canonical}
Fix integers $n\geq1$ and $r\geq0$, and a time horizon $T>0$.  Let $Q_n\succeq0$,
$\Sigma_0,\Sigma_T\succ0$, and prescribe arbitrary endpoint jets
\[
  U_n^{(j)}(0),\; U_n^{(j)}(T),
 \; j=0,\ldots,r,
\]
for a vector-valued control $U_n:[0,T]\to\R^n$.
Then there exist $C^\infty$-functions
$\Sigma:[0,T]\to\Spp$ (the $n\times n$ positive definite matrices) and $U_n:[0,T]\to\R^n$ such that
\begin{equation}
  \dot\Sigma(t)
  =A_n\Sigma(t)+\Sigma(t) A_n^{\mathsf T}
   +B_nU_n(t)^{\mathsf T}+U_n(t)B_n^{\mathsf T}+Q_n,
  \label{eq:canonical-lyap}
\end{equation}
with
\[
  \Sigma(0)=\Sigma_0,
  \qquad
  \Sigma(T)=\Sigma_T,
\]
and the prescribed endpoint jets of $U_n(t)$.
\end{proposition}

\begin{proof}
We prove the claim for all $r\geq0$ by induction on $n$.
For $n=1$, equation \eqref{eq:canonical-lyap} is
\begin{equation}
  \dot\sigma=2u+q,
  \qquad q\geq0.
  \label{eq:scalar}
\end{equation}
Prescribing $u^{(j)}$ at the endpoints, for $j=0,\ldots,r$, is
equivalent to prescribing the endpoint derivatives of $\sigma$ up to
order $r+1$, via
\[
  \sigma^{(j+1)}
  =2u^{(j)}+\delta_{j0}q.
\]
Lemma~\ref{lem:positive-jets} gives a positive smooth function
$\sigma$ with the prescribed endpoint values and jets, and then
$u=(\dot\sigma-q)/2$ has the desired endpoint jets.

Assume now that the result holds in dimension $n-1$, for every order of
endpoint jets.  Partition
\begin{equation}
  \Sigma=
  \begin{bmatrix}
    P & p\\
    p^{\mathsf T} & p_0
  \end{bmatrix},
  \qquad
  Q_n=
  \begin{bmatrix}
    Q_1 & q\\
    q^{\mathsf T} & q_0
  \end{bmatrix},
  \qquad
  U_n=
  \begin{bmatrix}
    w\\
    w_0
  \end{bmatrix},
  \label{eq:block-data}
\end{equation}
where $P,Q_1\in\mathbb S^{n-1}$, $P\succ 0$, $Q_1\succeq 0$, $p,q,w\in\R^{n-1}$, and $p_0,q_0,w_0$
are scalar.  Equation \eqref{eq:canonical-lyap} is equivalent to
\begin{subequations}
\label{eq:block-system}
\begin{align}
  \dot P
  &=A_{n-1}P+PA_{n-1}^{\mathsf T}
    +b_{n-1}p^{\mathsf T}+pb_{n-1}^{\mathsf T}+Q_1,
  \label{eq:block-P}\\
  \dot p
  &=A_{n-1}p+b_{n-1}p_0+w+q,
  \label{eq:block-z}\\
  \dot p_0
  &=2w_0+q_0.
  \label{eq:block-s}
\end{align}
\end{subequations}

The zeroth-order endpoint values of $p$ and $p_0$ are fixed by
$\Sigma_0$ and $\Sigma_T$.  The prescribed endpoint jets of $w_0$
determine the endpoint derivatives of $p_0$ recursively through order
$r+1$ by
\begin{equation}
  p_0^{(j+1)}=2w_0^{(j)}+\delta_{j0}q_0,
  \qquad j=0,\ldots,r.
  \label{eq:s-jets}
\end{equation}
Likewise, the prescribed endpoint jets of $w$, together with the already
determined jets of $p_0$, determine the endpoint derivatives of $p$
through order $r+1$ by
\begin{equation}
  p^{(j+1)}
  =w^{(j)}+A_{n-1}p^{(j)}+b_{n-1}p_0^{(j)}+\delta_{j0}q,
  \qquad j=0,\ldots,r.
  \label{eq:z-jets}
\end{equation}

Equation \eqref{eq:block-P} is precisely the
$(n-1)$-dimensional canonical Lyapunov equation, with $p$ as its control.
Apply the induction hypothesis in dimension $n-1$ with endpoint jets of
$p$ prescribed through order $r+1$.  We obtain smooth functions
$P(t)\succ0$ and $p(t)$ satisfying \eqref{eq:block-P}, the required
endpoint conditions, and the prescribed endpoint jets of $p$.

Define
\[
  \phi(t):=p(t)^{\mathsf T}P(t)^{-1}p(t).
\]
Since the prescribed endpoint matrices $\Sigma_0$ and $\Sigma_T$ are
positive definite, their Schur complements give
\[
  p_0(0)-\phi(0)>0,
  \qquad
  p_0(T)-\phi(T)>0.
\]
The desired endpoint jets of $p_0$ and the already constructed function
$\phi$ determine endpoint jets for $d:=p_0-\phi$, with positive zeroth
order values.  By Lemma~\ref{lem:positive-jets}, choose a smooth
$d(t)>0$ with those endpoint jets and set
\begin{equation}
  p_0(t):=\phi(t)+d(t).
  \label{eq:choose-s}
\end{equation}
Then $p_0$ has all the desired endpoint jets and, by the Schur complement
criterion,
\[
  \Sigma(t)=
  \begin{bmatrix}
    P(t) & p(t)\\
    p(t)^{\mathsf T} & p_0(t)
  \end{bmatrix}
  \succ0
 \mbox{ for }t\in[0,T].
\]
Finally, define
\begin{equation}
  w:=\dot p-A_{n-1}p-b_{n-1}p_0-q,
  \quad
  w_0:=\frac{\dot p_0-q_0}{2}.
  \label{eq:recover-control}
\end{equation}
Equations \eqref{eq:s-jets}--\eqref{eq:z-jets} ensure that $w$ and
$w_0$ have exactly the prescribed endpoint jets.  Thus
\eqref{eq:block-system} holds and the induction is complete.
\end{proof}

\begin{remark}
Only the case $r=0$ is needed to complete the proof of Theorem~3.  The
higher-order jet statement is introduced because the $n$-dimensional
control depends on one derivative of the control in the
$(n-1)$-dimensional problem; establishing the claim simultaneously for
all finite jet orders makes the induction complete without any loss of
smoothness.
\end{remark}

\section{Reduction to a general controllable pair}

Proposition~\ref{prop:canonical} can be readily extended to a general controllable pair $(A,B)$, following the argument given in \cite{ChenGeorgiouPavonPartII}.  Specifically, by Heymann's lemma
\cite{Hautus1977}, one may choose $v\in\R^m$ and
$K_0\in\R^{m\times n}$ so that
\[
  (A-BK_0,Bv)
\]
is controllable.  Single-input pole placement then provides a row vector
$\kappa^{\mathsf T}$ such that
\[
  A-BK_0-(Bv)\kappa^{\mathsf T}
\]
is nilpotent.  Upon setting $K:=K_0+v\kappa^{\mathsf T}$, the pair
$(A-BK,Bv)$ remains controllable with nilpotent state matrix.  Since a
controllable single-input nilpotent pair is similar to
\eqref{eq:canonical-pair}, Proposition~\ref{prop:canonical}
provides, after a similarity transformation and transformation back, a
smooth positive-definite path satisfying
\begin{equation}
  \dot\Sigma
  =(A-BK)\Sigma+\Sigma(A-BK)^{\mathsf T}
   +(Bv)w^{\mathsf T}+w(Bv)^{\mathsf T}+Q
  \label{eq:feedback-reduced}
\end{equation}
for some smooth $w:[0,T]\to\R^n$, with the prescribed covariance
endpoints.
Set
\begin{equation}
  U:=wv^{\mathsf T}-\Sigma K^{\mathsf T}.
  \label{eq:lift-control}
\end{equation}
Then
\[
  BU^{\mathsf T}+UB^{\mathsf T}
  =Bvw^{\mathsf T}+wv^{\mathsf T}B^{\mathsf T}
   -BK\Sigma-\Sigma K^{\mathsf T}B^{\mathsf T}.
\]
Substitution into \eqref{eq:feedback-reduced} gives exactly
\eqref{eq:lyap}.  Hence the desired positive-definite steering property
holds for every controllable pair $(A,B)$.

\section{Conclusion}

The statement \cite[Theorem~3]{ChenGeorgiouPavonPartII} is correct.  The issue in the original proof was the
unsupported claim that arbitrary endpoint values of the control
can be specified.  Proposition~\ref{prop:canonical}, builds the endpoint
conditions into a strengthened induction on endpoint jets. The authors are indebted to Fengjiao Liu and Panagiotis Tsiotras for pointing out the issue in the original argument. 

\bibliographystyle{IEEEtran}
\bibliography{refs}

\end{document}